\documentclass[12pt, reqno]{amsart}

\usepackage{fullpage}
\usepackage{amssymb}
\usepackage{url}

\theoremstyle{plain}
\newtheorem{theorem}{Theorem}[section]
\newtheorem{lemma}[theorem]{Lemma}
\newtheorem{problem}[theorem]{Problem}
\newtheorem*{acknowledgment}{Acknowledgment}

\numberwithin{equation}{section}

\newcommand{\by}[1]{\overset{#1}=}             
\newcommand{\byeq}[1]{\by{\eqref{#1}}}         

\newcommand{\GRB}{\mathbf{GRB}}
\newcommand{\I}{\mathbf{I}}

\title{The solution of an open problem on semigroup inclusion classes}

\author{Maria Leonor Ara\'{u}jo}
\author{Maria Teresa Ara\'{u}jo}
\author{Michael Kinyon${}^*$}
\thanks{${}^*$ Partially supported by Simons Foundation Collaboration Grant 359872 and
by the Fundaç\~{a}o para a Ci\^{e}ncia e a Tecnologia (Portuguese Foundation for Science and Technology)
through the project PTDC/MAT-PUR/31174/2017}

\address[M.L. Ara\'{u}jo, M.T. Ara\'{u}jo]{Col\'egio Mira Rio, Estrada de Telheiras n 113, 1600--768 Lisboa, Portugal}
\email[M.L. Ara\'{u}jo]{\url{Maria.Araujo.20@colegiomirario.pt}}
\email[M.T. Ara\'{u}jo]{\url{Maria.Teresa.20@colegiomirario.pt}}

\address[Kinyon]{Department of Mathematics, University of Denver, Denver, CO 80208, USA}
\address[Kinyon]{Centre for Mathematics and Applications, Faculdade de Ci\^{e}ncias e Tecnologia,
Universidade Nova de Lisboa, Campus da Caparica, 2829-516 Caparica, PT}
\email[Kinyon]{mkinyon@du.edu}

\begin{document}

\begin{abstract}
The semigroup inclusion class $\I = [xyxy = xy; xyz \in \{xywz, xuyz\}]$ is the union
of two maximal subvarieties of $\GRB= [xyzxy=xy]$.
Monzo \cite{monzo} described the lattice of semigroup inclusion classes below $\I$ and
asked if $\I$ is covered by $\GRB$. Our main result is a characterization of $\I$
which makes it easy to answer Monzo's question in the negative.
\end{abstract}

\maketitle

\section{Introduction}

A semigroup $S$ is a \emph{rectangular band} if it satisfies the identity $xyx = x$, or equivalently, the pair of
identities $xyz = xz$ and $x^2 = x$.
Throughout this paper we work in the variety $\GRB$, the class of all semigroups $S$ such that $S^2 = \{ xy\mid x,y\in S\}$ is
a rectangular band \cite{monzo}. (We will prove below that $\GRB= [xyzxy=xy]$.)

Let $\I$ denote the union of the two subvarieties
$\GRB\cap [xywz = xyz]$ and $\GRB\cap [xuyz = xyz]$. The class $\I$ is not a variety, it is an \emph{inclusion class}
\cite{lyapin,monzo2008} axiomatized as follows $$\I = [xyxy = xy; xyz \in \{xywz, xuyz\}] = \GRB\cap [xyz \in \{xywz, xuyz\}].$$

Monzo \cite{monzo} completely described the lattice of inclusion classes below $\I$. In the open questions at the
end of his paper, he posed the following

\begin{problem}[\cite{monzo}]\label{Prb:monzo}
  Does $\GRB$ cover $I$?
\end{problem}

The purpose of this note is to answer Monzo's question in the negative. This is made easier
by our main result, which is a characterization of the class $\I$:

\begin{theorem}\label{characterisation}
Let $S$ be a semigroup in $\GRB$. The following are equivalent:
\begin{enumerate}
\item For all $x,y,z,u\in S$, $xyz \in \{xywz, xuyz\}$;
\item For all $x,y,z\in S$, $xy^2 = xy$ or $z^2 x = zx$.
\end{enumerate}
\end{theorem}

From this theorem we inferred  the inclusion class  $xy^2 = xy$ or $y^2 x = yx$, for all $x,y\in S$, a class that properly contains $\I$ and is properly contained in $\GRB$, thus solving  Problem \ref{Prb:monzo}.
\section{Proof of Theorem \ref{characterisation}}

For ease of reference, we collect identities that hold in $\GRB$. In the list below, either
\eqref{Eq:GRB2} or the pair \eqref{Eq:GRB0}, \eqref{Eq:GRB1}
could be taken as an equational base of axioms for $\GRB$.

\begin{lemma}\label{Lem:GRB_identities}
The following identities hold in $\GRB$.
\begin{align}
   xyxy &= xy       \label{Eq:GRB0} \\
  xyzuv &= xyuv     \label{Eq:GRB1} \\
  xyzxy &= xy       \label{Eq:GRB2} \\
   xyyz &= xyz      \label{Eq:GRB3} \\
   xxx &= xx        \label{Eq:GRB4}\,.
\end{align}
\end{lemma}
\begin{proof}
For $S$ a semigroup in $\GRB$, the identity \eqref{Eq:GRB0} holds since $S^2$ is a band.
We have $xy\cdot z\cdot uv \byeq{Eq:GRB0} xy(zuv)uv = xyuv$ since $S^2$ is a rectangular band,
so \eqref{Eq:GRB1} holds. Then \eqref{Eq:GRB2} follows immediately from \eqref{Eq:GRB0} and
\eqref{Eq:GRB1}. For \eqref{Eq:GRB3}, we have $xyyz \byeq{Eq:GRB1} x\underbrace{yzxyyz} \byeq{Eq:GRB2} xyz$.
Finally, for \eqref{Eq:GRB4}, we have $xxx \byeq{Eq:GRB3} xxxx \byeq{Eq:GRB0} xx$.
\end{proof}

The first step in our proof of Theorem \ref{characterisation} is to show that (1) implies (2).

\begin{lemma}\label{lemma18}
Let $S$ be a semigroup in $\GRB$ such that for all $x,y,z,u,w\in S$, $xyz \in \{xywz, xuyz\}$.
Then for all $x,y,z\in S$, $xy^2 = xy$ or $z^2 x = zx$.
\end{lemma}
\begin{proof}
First we prove that for all $x,y,z,u,w,v\in S$,
\begin{equation}\label{Eq:18}
xyzu = xzu\qquad\text{or}\qquad vzwu = vzu.
\end{equation}
Indeed, suppose $xyzu \neq xzu$, that is,  $abcd \neq acd$ for some $a,b,c,d\in S$. We claim that $vcwd=vcd$. In fact, by (\ref{Eq:18}), if $abcd \neq acd$, then
$acwd = acd$ for all $w\in S$ (*). Thus for all $v,w\in S$,
\[
\underbrace{vcw}d \byeq{Eq:GRB3} \underbrace{vccw}d \byeq{Eq:GRB2} vc\underbrace{acwd}
\stackrel{(*)}{=} vcacd \byeq{Eq:GRB2} vccd \byeq{Eq:GRB3} vcd\,,
\]
as claimed.

Now in \eqref{Eq:18}, set $x = z$, $u = y$ and simplify to get
$zy = zzy$ or $vzwy = vzy$ for all $y,z,v,w\in S$. For convenience,
rename $y \to x$ so that
\begin{equation}\label{Eq:tmp1}
  zx = zzx\qquad\text{or}\qquad vzwx = vzx
\end{equation}
for all $x,z,v,w\in S$.

Again in \eqref{Eq:18}, set $w = v$, $u = z$ and simplify to get
$xyzz = xzz$ or $vz = vzz$ for all $x,y,z,v\in S$. Now for all $u\in S$,
$xyzu \byeq{Eq:GRB3} xyzzu$ and $xzzu \byeq{Eq:GRB3} xzu$, and so we have
$xyzu = xzu$ or $vz = vzz$ for all $x,y,z,u,v\in S$. For convenience,
rename $v \to x$, $z\to y$, $x\to s$, $y\to u$, $u\to t$ to get
\begin{equation}\label{Eq:tmp2}
  suyt = syt\qquad\text{or}\qquad xy = xyy
\end{equation}
for all $x,y,s,t,u\in S$.

Now suppose $xy^2 \neq xy$ and $z^2 x \neq zx$ for some $x,y,z\in S$, that is, $ab^2 \neq ab$ and $c^2 a \neq ca$ for some $a,b,c\in S$ (**).
Then by \eqref{Eq:tmp1} and \eqref{Eq:tmp2}, we get
\begin{align}
  vcwa &= vca \label{Eq:tmp3} \\
  subt &= sbt \label{Eq:tmp4}
\end{align}
for all $v,w,s,t\in S$. Thus
\begin{equation}\label{Eq:tmp5}
sba \byeq{Eq:tmp4} scba \byeq{Eq:tmp3} sca
\end{equation}
for all $s\in S$. Now we compute
\[
ca \byeq{Eq:GRB0} caca \byeq{Eq:tmp5} caba \byeq{Eq:tmp4} cba \byeq{Eq:tmp5} cca\,.
\]
This is a contradiction with (**). Therefore for all $x,y,z\in S$, $xy^2 = xy$
or $z^2 x = zx$.
\end{proof}

This proves  (1) implies (2) of Theorem \ref{characterisation}.
We turn now to the converse.

\begin{lemma}\label{converse}
Let $S$ be a semigroup in $\GRB$ such that for all $x,y,z\in S$, $xy^2 = xy$ or $z^2 x = zx$.
Then for all $x,y,z,u,w\in S$, $xyz \in \{xywz, xuyz\}$.
\end{lemma}
\begin{proof}
Assume that for some $a,b,c,d,e\in S$, $abc\neq abdc$ and $abc \neq aebc$.

First, if $xyy\neq xy$ for some $x,y\in S$, then for all $z\in S$, $zzx=zx$ and thus for all $z,u,w\in S$,
$zxuw = zzxuw \byeq{Eq:GRB2} zzuw$. Therefore for all $x,y,z,u,w\in S$,
\begin{equation}\label{Eq:24}
xyy = xy \quad\text{or}\quad zxuw = zzuw\,.
\end{equation}

By assumption $uxyy = uxy$ or $zzux = zux$ for all $x,y,z,u\in S$, so if $zzux\neq zux$ for some
$x,z,u\in S$, then for all $y\in S$, $uxyy = uxy$, and so
\[
xyy \byeq{Eq:GRB4} xyyy \byeq{Eq:GRB1} xyuxyy = xyuxy \byeq{Eq:GRB2} xyxy \byeq{Eq:GRB0} xy\,.
\]
Therefore for all $x,y,z,u\in S$,
\begin{equation}\label{Eq:27}
xyy = xy \quad\text{or}\quad zzux = zux\,.
\end{equation}

Now if $xyy\neq xy$ for some $x,y\in S$, then from \eqref{Eq:24} and \eqref{Eq:27},
\begin{align}
z_1 x u_1 w_1 &= z_1 z_1 u_1 w_1 \label{Eq:41a} \\
z_2 z_2 u_2 x &= z_2 u_2 x \label{Eq:41b}
\end{align}
for all $z_i,u_i,w_i\in S$. Thus for all $z,u,w\in S$,
\[
zxuw \byeq{Eq:GRB1} zxuxuw \byeq{Eq:41a} zzuxuw \byeq{Eq:41b} zuxuw \byeq{Eq:GRB1} zuuw \byeq{Eq:GRB3} zuw\,,
\]
and also
\[
zuxw \byeq{Eq:41b} zzuxw \byeq{Eq:GRB1} zzxw \byeq{Eq:GRB1} zzxxw \byeq{Eq:41b} zxxw \byeq{Eq:GRB3} zxw\,.
\]
Therefore
\begin{align}
z_1 u_1 x w_1 &= z_1 x w_1 \label{Eq:45b}\\
z_2 x u_2 w_2 &= z_2 u_2 w_2 \label{Eq:45a}
\end{align}
for all $z_i,u,v,w_i\in S$. Thus for all $z,u,v,w\in S$,
\[
zuvw \byeq{Eq:GRB2} zuxvw \byeq{Eq:45b} zxvw \byeq{Eq:45a} zvw\,.
\]
However, we have assumed that $aebc\neq abc$. It follows that $xyy=xy$ for all $x,y\in S$.
But then $abdc = abdcc \byeq{Eq:GRB2} abcc = abc$ which is a contradiction. This completes the proof.
\end{proof}

Lemmas \ref{lemma18} and \ref{converse} together complete the proof of Theorem \ref{characterisation}.

\section{Examples}

Now we answer Problem \ref{Prb:monzo}.
First we note that there is no variety of semigroups properly between $\I$ and $\GRB$.
This is because the two subvarieties $\GRB\cap [xywz = xyz]$ and $\GRB\cap [xuyz = xyz]$
are maximal in the lattice of subvarieties of $\GRB$ \cite{monzo_correction}. Thus the
variety generated by any semigroup in $\GRB$ which is not in $\I$ must be $\GRB$ itself.

Therefore a negative answer to Problem \ref{Prb:monzo} involves finding a class
properly contained between $\I$ and $\GRB$. Theorem \ref{characterisation} suggests how to do this.
Consider the class
\[
  \mathbf{A} = \GRB\cap [ xyy=xy \quad\text{or}\quad yyx = yx ]\,.
\]
It is obvious that $\I$ is contained in $ \mathbf{A}$.
The following is a smallest semigroup in $\mathbf{A}$ but not in $\I$:
\[
\begin{array}{c|ccccccccccc}
\cdot & 0 & 1 & 2 & 3 & 4 & 5 & 6 & 7 & 8 & 9 & 10\\
\hline
    0 & 0 & 3 & 0 & 3 & 4 & 4 & 3 & 0 & 4 & 0 & 3 \\
    1 & 0 & 4 & 0 & 3 & 4 & 4 & 3 & 0 & 4 & 0 & 3 \\
    2 & 7 & 6 & 9 & 6 & 5 & 8 & 10 & 9 & 8 & 9 & 10 \\
    3 & 0 & 4 & 0 & 3 & 4 & 4 & 3 & 0 & 4 & 0 & 3 \\
    4 & 0 & 4 & 0 & 3 & 4 & 4 & 3 & 0 & 4 & 0 & 3 \\
    5 & 7 & 5 & 7 & 6 & 5 & 5 & 6 & 7 & 5 & 7 & 6 \\
    6 & 7 & 5 & 7 & 6 & 5 & 5 & 6 & 7 & 5 & 7 & 6 \\
    7 & 7 & 6 & 7 & 6 & 5 & 5 & 6 & 7 & 5 & 7 & 6 \\
    8 & 9 & 8 & 9 & 10 & 8 & 8 & 10 & 9 & 8 & 9 & 10 \\
    9 & 9 & 10 & 9 & 10 & 8 & 8 & 10 & 9 & 8 & 9 & 10 \\
    10 & 9 & 8 & 9 & 10 & 8 & 8 & 10 & 9 & 8 & 9 & 10
\end{array}
\]

The following is a smallest semigroup in $\GRB$ but not in $\mathbf{A}$:
\[
\begin{array}{c|cccccccccc}
\cdot & 0 & 1 & 2 & 3 & 4 & 5 & 6 & 7 & 8 & 9\\
\hline
    0 & 0 & 2 & 2 & 8 & 2 & 0 & 0 & 2 & 8 & 8 \\
    1 & 5 & 3 & 4 & 3 & 7 & 6 & 6 & 7 & 9 & 3 \\
    2 & 0 & 8 & 2 & 8 & 2 & 0 & 0 & 2 & 8 & 8 \\
    3 & 6 & 3 & 7 & 3 & 7 & 6 & 6 & 7 & 3 & 3 \\
    4 & 5 & 9 & 4 & 9 & 4 & 5 & 5 & 4 & 9 & 9 \\
    5 & 5 & 4 & 4 & 9 & 4 & 5 & 5 & 4 & 9 & 9 \\
    6 & 6 & 7 & 7 & 3 & 7 & 6 & 6 & 7 & 3 & 3 \\
    7 & 6 & 3 & 7 & 3 & 7 & 6 & 6 & 7 & 3 & 3 \\
    8 & 0 & 8 & 2 & 8 & 2 & 0 & 0 & 2 & 8 & 8 \\
    9 & 5 & 9 & 4 & 9 & 4 & 5 & 5 & 4 & 9 & 9
\end{array}
\]

\medskip

This suggests the following problem, with which we conclude this paper.

\begin{problem}
  Classify the semigroup inclusion classes properly between $\I$ and $\GRB$.
\end{problem}

\begin{acknowledgment}
We thank Edmond Lee for calling our attention to \cite{monzo_correction}. We are pleased to acknowledge the use of the automated theorem prover \texttt{Prover9}
and the finite model builder \texttt{Mace4}, both developed by McCune \cite{mccune}.
\end{acknowledgment}

\end{document}